\documentclass[12pt,twoside]{amsart}
\usepackage{amssymb,amscd}
\usepackage{mathrsfs}
\usepackage{array,float}

\theoremstyle{plain}
\newtheorem{thm}{Theorem}[section]
\newtheorem{lem}[thm]{Lemma}
\newtheorem{pro}[thm]{Proposition}
\newtheorem{cor}[thm]{Corollary}

\theoremstyle{remark}\theoremstyle{remark}
\newtheorem*{rem*}{Remark}

\theoremstyle{definition}

\numberwithin{equation}{section}

\newcommand{\Z}{\mathbb{Z}}   
\newcommand{\Q}{\mathbb{Q}}
\newcommand{\N}{\mathbb{N}}

\DeclareMathOperator{\ad}{ad}
\DeclareMathOperator{\id}{id}
\DeclareMathOperator{\Cen}{C}

\newcommand{\ackn}{  \noindent{\sc Acknowledgement }\hspace{5pt} }

\begin{document}

\title[Pro-$p$ groups with constant generating number]{Pro-$p$ groups
  with constant generating number on open subgroups}

\author{Benjamin Klopsch}
\address{Department of Mathematics \\
  Royal Holloway, University of London \\
  Egham TW20 0EX, UK}
\email{Benjamin.Klopsch@rhul.ac.uk}

\author{Ilir Snopce}
\address{Department of Mathematics\\
  University of Virginia \\
  Charlottesville, Virginia 22904-4137, USA }
\email{snopce@math.binghamton.edu}

\begin{abstract}
  Let $p$ be a prime.  We classify finitely generated pro-$p$ groups
  $G$ which satisfy $d(H) = d(G)$ for all open subgroups $H$ of $G$.
  Here $d(H)$ denotes the minimal number of topological generators for
  the subgroup~$H$.  Within the category of $p$-adic analytic pro-$p$
  groups, this answers a question of Iwasawa.
\end{abstract}


\maketitle

\section{Introduction}

Let $p$ be a prime.  Motivated by a question of Kenkichi Iwasawa,
which we discuss below, we consider profinite groups $G$ satisfying
the condition
\begin{equation}\tag{$*$}\label{equ:star}
  d(H) = d(G) \quad \text{for all open subgroups $H$ of $G$.}
\end{equation}
Here $d(H)$ denotes the minimal number of topological generators for
the subgroup~$H$.  For lack of a shorter description, we say that
profinite groups $G$ with the property \eqref{equ:star} have
\emph{constant generating number on open subgroups}.  The aim of this
paper is to give a complete classification of finitely generated
pro-$p$ groups with constant generating number on open subgroups.

\begin{thm}\label{thm:main_theorem}
  Let $G$ be a finitely generated pro-$p$ group and let $d :=
  d(G)$. Then $G$ has constant generating number on open subgroups if
  and only if it is isomorphic to one of the groups in the following
  list:

  \begin{enumerate}
  \item the abelian group $\Z_p^d$, for $d \geq 0$;
  \item the metabelian group $\langle y \rangle \ltimes A$, for $d
    \geq 2$, where $\langle y \rangle \cong \Z_p$, $A \cong
    \Z_p^{d-1}$ and $y$ acts on $A$ as scalar multiplication by
    $\lambda$, with $\lambda = 1+p^s$ for some $s \geq 1$, if $p>2$,
    and $\lambda = \pm (1 + 2^s)$ for some $s \geq 2$, if $p=2$;
  \item the group $\langle w \rangle \ltimes B$ of maximal class, for
    $p=3$ and $d=2$, where $\langle w \rangle \cong C_3$, $B = \Z_3 +
    \Z_3 \omega \cong \Z_3^2$ for a primitive $3$rd root of unity
    $\omega$ and where $w$ acts on $B$ as multiplication by $\omega$;
  \item the metabelian group $\langle y \rangle \ltimes A$, for $p=2$
    and $d \geq 2$, where $\langle y \rangle \cong \Z_2$, $A \cong
    \Z_2^{d-1}$ and $y$ acts on $A$ as scalar multiplication by $-1$.
   \end{enumerate}
\end{thm}

This generalises the partial results in \cite{Sn09}, which were
obtained by quite different methods.  Thematically our work is linked
to other recent results on generating numbers of pro-$p$ groups; e.g.\
see~\cite{GoKl_pre,Kl_pre}.  The metabelian groups resulting from our
classification can easily be realised as subgroups of the affine group
$\textup{Aff}_d(\Z_p) = \textup{GL}_d(\Z_p) \ltimes \Z_p^d$.


In view of our theorem, it is an interesting problem to investigate
more general classes of finitely generated profinite groups with
respect to the property \eqref{equ:star}, e.g.\ compact $p$-adic
analytic groups.  We conclude this introduction with a short
motivation and a sketch of the proof of
Theorem~\ref{thm:main_theorem}.

\smallskip

It is well known that the (topological) Schreier index formula can be
used to characterise free pro-$p$ groups: a finitely generated pro-$p$
group $G$ is a free pro-$p$ group if and only if $d(H)-1 = \lvert G :
H \rvert (d(G)-1)$ for every open subgroup $H$ of $G$;
cf.~\cite{Lu82}.  For every positive integer $n$, let $\mathscr{E}_n$
denote the class of all finitely generated pro-$p$ groups $G$
satisfying:
\begin{equation*}
  d(H)-n = \lvert G : H \rvert (d(G)-n) \quad \text{for all open
  subgroups $H$ of $G$}.
\end{equation*}
For instance, the class $\mathscr{E}_1$ consists precisely of the
finitely generated free pro-$p$ groups.  As described in \cite{DuLa83},
Iwasawa observed that, for fixed $n$, pro-$p$ groups belonging to the
class $\mathscr{E}_n$ have interesting representation-theoretic
properties and he raised the question of determining the groups
belonging to $\mathscr{E}_n$ for each $n > 1$.  In \cite{DuLa83},
Dummit and Labute answer Iwasawa's question for $n=2$ in the case of
one-relator groups: $G$ is a Demushkin group if and only if $G$ is a
one-relator group and $G \in \mathscr{E}_2$.

Note that $\mathscr{E}_n$ is non-empty for every $n$ because, clearly,
it contains the free abelian pro-$p$ group $\Z_p^n$ of rank $n$.  In
\cite{Ya07}, Yamagishi remarked that no other examples are known to
him when $n \geq 3$.  Recently, for $p>3$, the second author
determined all $p$-adic analytic pro-$p$ groups that belong to the
class $\mathscr{E}_3$; see \cite{Sn09}.  His approach relied on the
classification of $3$-dimensional soluble $\Z_p$-Lie lattices provided
in \cite{GoKl09}.

Theorem~\ref{thm:main_theorem} answers completely the question of
Iwasawa for pro-$p$ groups of finite rank, or equivalently $p$-adic
analytic pro-$p$ groups.  Indeed, it is easy to see that such a group
$G$ belongs to $\mathscr{E}_n$ if and only if it has constant
generating number on open subgroups and $n = \dim(G)$.  Here $\dim(G)$
denotes the dimension of $G$ as a $p$-adic Lie group.  We record this
consequence as

\begin{cor}
  Let $G$ be a $p$-adic analytic pro-$p$ group and let $n \in \N$.
  Then $G$ belongs to the class $\mathscr{E}_n$ if and only if $G$ is
  isomorphic to one of the groups listed in
  Theorem~\ref{thm:main_theorem} and $n = \dim(G)$.
\end{cor}

\smallskip

Our proof of Theorem~\ref{thm:main_theorem} proceeds as follows.  By a
Lie theoretic argument, we first deal with saturable pro-$p$ groups:
in this special situation we only obtain groups listed in (1) and (2)
of the theorem; see Corollary~\ref{cor:sat_groups}.  A general
$p$-adic analytic pro-$p$ group $G$ contains a saturable normal
subgroup of finite index.  By induction on the index, it suffices to
study the case where $G$ contains an open normal subgroup $H$ of
`known type' -- e.g.\ saturable, and hence listed in (1) or (2) of the
theorem -- such that $G/H \cong C_p$.  If $H$ is abelian, we consider
$H$ as a $\Z_p[G/H]$-module, and the minimal number of generators
$d(\mathcal{M})$ for the $\Q_p[G/H]$-module $\mathcal{M} := \Q_p
\otimes H$ becomes a key invariant.  If $H$ is not abelian, we
consider the actions of $G/H$ and $H/A$ on an abelian characteristic
subgroup $A$ of co-dimension $1$ in $H$.  While the details of the
argument work out smoothly for $p \geq 5$, a more careful analysis is
required to deal the small primes $p=2$ and $p=3$.

\bigskip

\noindent \textit{Notation.}  Throughout the paper, $p$ denotes a
prime.  We write $\N$ for the set of natural numbers and $\N_0 := \N
\cup \{0\}$.  The $p$-adic integers and $p$-adic numbers are denoted
by $\Z_p$ and $\Q_p$.  The minimal number of generators of a group $G$
is denoted by $d(G)$, and a similar notation is employed for other
algebraic structures, such as Lie algebras, Lie lattices and modules.
Moreover, we tacitly interpret generators as topological generators as
appropriate.


\section{Saturable pro-$p$ groups and preliminaries}

Let $G$ be a finitely generated pro-$p$ group with constant generating
number on open subgroups.  Then $G$ has finite rank and admits the
structure of a $p$-adic analytic group (in a unique way); cf.\
\cite[Chapters 8 and 9]{DidSMaSe99}.  There is a normal open subgroup
$H$ of $G$ which is uniformly powerful or, more generally, saturable.
This group $H$ corresponds via Lie theory to a $\Z_p$-Lie lattice
$L_H$ which is powerful (respectively saturable); see
\cite{DidSMaSe99,Kl05,GoKl09}.  At the level of saturable pro-$p$
groups the group theoretic property we are interested in translates
readily to an equivalent condition on the associated Lie lattice.

\begin{pro}\label{pro:sat_groups}
  Let $G$ be a saturable pro-$p$ group with associated $\Z_p$-Lie
  lattice $L = L_G$.  Then $G$ has constant generating number on open
  subgroups if and only if the $\Q_p$-Lie algebra $\mathcal{L} := \Q_p
  \otimes_{\Z_p} L$ satisfies $d(\mathcal{L}) = \dim_{\Q_p}
  \mathcal{L}$.
\end{pro}

The proof of this plausible result is not completely straightforward,
because subgroups of saturable pro-$p$ groups are not necessarily
saturable.  We begin by proving an auxiliary lemma.

\begin{lem}
  Let $\mathcal{L}$ be a finite-dimensional Lie algebra over a field
  $F$ such that $d(\mathcal{L}) = \dim_F(\mathcal{L})$.  Then either
  $\mathcal{L}$ is abelian or $\mathcal{L}$ is metabelian of the form
  $\mathcal{L} = F x \oplus \mathcal{A}$, where $\mathcal{A}$ is an
  abelian ideal of co-dimension $1$ in $\mathcal{L}$ and $\ad(x)
  \vert_{\mathcal{A}} = \id_{\mathcal{A}}$.
\end{lem}

\begin{proof}
  Suppose that $\mathcal{L}$ is not abelian.  By
  \cite[Lemma~4.3]{LuMa87}, the Lie algebra $\mathcal{L}$ is
  metabelian, i.e.\ $[\mathcal{L},\mathcal{L}]$ is abelian.  Let
  $\mathcal{A}$ be a maximal abelian Lie subalgebra of $\mathcal{L}$
  containing $[\mathcal{L},\mathcal{L}]$.  Then $\mathcal{A}$ is an
  ideal of $\mathcal{L}$ and equal to its centraliser
  $\Cen_{\mathcal{L}}(\mathcal{A})$ in $\mathcal{L}$.  Because
  $d(\mathcal{L}) = \dim_F(\mathcal{L})$, it is clear that for every
  $x \in \mathcal{L}$ the action of $\ad(x)$ on $\mathcal{A}$ is
  scalar, and we obtain a representation $\rho: \mathcal{L}
  \rightarrow \mathfrak{gl}_1(F)$.  Since
  $\Cen_{\mathcal{L}}(\mathcal{A}) = \mathcal{A} \lneqq \mathcal{L}$,
  the sequence
  \begin{equation*}
    \begin{CD}
      0 @>>> \mathcal{A} @>\text{incl.}>> \mathcal{L} @>\rho>>
      \mathfrak{gl}_1(F) @>>> 0,
    \end{CD}
  \end{equation*}
  is exact, and we find $x \in \mathcal{L} \setminus \mathcal{A}$ such
  that $\mathcal{L} = F x \oplus \mathcal{A}$ and $\ad(x)
  \vert_{\mathcal{A}} = \id_{\mathcal{A}}$.
\end{proof}

\begin{cor}\label{cor:Lie_lattice}
  Let $L$ be a $\Z_p$-Lie lattice of dimension $d$ such that $d(\Q_p
  \otimes_{\Z_p} L) = d$.  Then one of the following holds:
 \begin{enumerate}
 \item[(i)] $L \cong \Z_p^d$ is abelian;
 \item[(ii)] $L = \Z_p x \oplus A$ where $A \cong \Z_p^{d-1}$ is an
   abelian ideal of $L$ and $\ad(x)$ acts on $A$ as multiplication by
   $p^s$ for suitable $s \in \N_0$.
 \end{enumerate}
\end{cor}

\begin{proof}[Proof of Proposition~\ref{pro:sat_groups}]
  Put $d := \dim_{\Q_p} \mathcal{L}$, and suppose that $G$ has
  constant generating number on open subgroups.  For a contradiction
  assume that $d(\mathcal{L}) < d$.  Then $L$ admits an open Lie
  sublattice $M$ which requires less than $d$ generators, say $a_1,
  \ldots, a_r$ where $r<d$.  The saturable group $G$ is recovered from
  the $\Z_p$-Lie lattice $L$ by defining a group multiplication on the
  set $L$ via the Hausdorff series.  We may thus consider the subgroup
  $H := \langle a_1, \ldots, a_r \rangle$ of $G$.  To arrive at the
  required contradiction, it suffices to show that $H$ is open in $G$.
  We find an open saturable subgroup $K$ of $H$ and $k \in \N$ such
  that $a_1^{p^k}, \ldots, a_r^{p^k} \in K$.  The $\Z_p$-Lie lattice
  associated to $K$ is naturally a Lie sublattice of $L$ and contains
  $p^k a_1, \ldots, p^ka_r$.  It is therefore of finite index in $M$
  and hence open in $L$.  Consequently, $K$ is open in~$G$.

  Conversely, suppose that $d(\mathcal{L}) = d$.  Then
  Corollary~\ref{cor:Lie_lattice} shows that $L$ is of a rather
  restricted form.  Clearly, if $L$ is abelian, so is $G$ and $G$ has
  constant generating number on open subgroups.  Suppose that $L$ is
  not abelian so that we have $L = \Z_p x \oplus A$ where $A \cong
  \Z_p^{d-1}$ is an abelian ideal of $L$ and $\ad(x)$ acts on $A$ as
  multiplication by $p^s$ for suitable $s \in \N_0$.  Since $L$ is
  saturable, we have $s \geq 1$ if $p>2$ and $s \geq 2$ if $p=2$; cf.\
  \cite[Section~2]{Kl05}.  Thus $G$ is isomorphic to the group
  $\langle y \rangle \ltimes A$, where $\langle y \rangle \cong \Z_p$
  and $y$ acts on $A$ as scalar multiplication by $1+p^s$.  One checks
  readily that the open subgroups of $G$ are essentially of the same
  form, with the parameter $s$ possibly taking larger values.  This
  explicit description shows that $G$ has constant generating number
  on open subgroups.
\end{proof}

\begin{cor}\label{cor:sat_groups}
  Let $G$ be a saturable pro-$p$ group of dimension $d$. Then $G$ has
  constant generating number on open subgroups if and only if one of
  the following holds:
 \begin{enumerate}
 \item[(i)] $G \cong \Z_p^d$ is abelian;
 \item[(ii)] $G \cong \langle y \rangle \ltimes A$, where $d \geq 2$,
   $\langle y \rangle \cong \Z_p$, $A \cong \Z_p^{d-1}$ and $y$ acts
   on $A$ as scalar multiplication by $1+p^s$ for some $s \geq 1$, if
   $p>2$, and $s \geq 2$, if $p=2$.
 \end{enumerate}
\end{cor}

We conclude this section with a technical lemma which plays a central
role in our proof of Theorem~\ref{thm:main_theorem}.

\begin{lem}\label{lem:module}
  Let $\langle z \rangle \cong C_p$ be a cyclic group of order $p$,
  and let $M$ be a finitely generated $\Z_p \langle z \rangle$-module
  which is free as a $\Z_p$-module.  Then $M$ decomposes as
  $$
  M \cong n_1 \cdot I_1 \oplus n_2 \cdot I_2 \oplus n_3 \cdot I_3,
  $$
  where $n_1, n_2, n_3 \in \N_0$ and $I_1$, $I_2$, $I_3$ are
  indecomposable $\Z_p \langle z \rangle$-modules of $\Z_p$-dimensions
  $1$, $p-1$ , $p$ so that $n_1 + (p-1) n_2 + p n_3 = \dim_{\Z_p} M$.

  If the $\Q_p \langle z \rangle$-module $\mathcal{M} := \Q_p \otimes
  M$ satisfies $1+ d(\mathcal{M}) \geq \dim_{\Q_p} \mathcal{M}$, then
  \begin{equation}\label{equ:module}
    1 + \max \{ 0, n_2 - n_1 \} \geq (p-1)(n_2 + n_3).
  \end{equation}
\end{lem}

\begin{proof}
  According to~\cite{HeRe62}, there are three types of indecomposable
  $\Z_p \langle z \rangle$-modules which are free as $\Z_p$-modules:
  \begin{enumerate}
  \item[(i)] the trivial module $I_1 = \Z_p$ of $\Z_p$-dimension $1$,
  \item[(ii)] the module $I_2 = \Z_p \langle z \rangle / (\Phi(z))$ of
    $\Z_p$-dimension $p-1$, where $\Phi$ denotes the $p$th cyclotomic
    polynomial,
  \item[(iii)] the free module $I_3 = \Z_p \langle z \rangle$ of
    $\Z_p$-dimension $p$.
  \end{enumerate}
  The module $M$ decomposes as described as a direct sum of
  indecomposable submodules.

  Note that $\mathcal{I}_1 := \Q_p \otimes_{\Z_p} I_1$ and
  $\mathcal{I}_2 := \Q_p \otimes_{\Z_p} I_2$ are irreducible $\Q_p
  \langle z \rangle$-modules, while $\mathcal{I}_3 := \Q_p
  \otimes_{\Z_p} I_3$ decomposes as $\mathcal{I}_3 \cong \mathcal{I}_1
  \oplus \mathcal{I}_2$.  From this it follows that $\mathcal{M}$
  decomposes as $\mathcal{M} \cong (n_1 + n_3) \cdot \mathcal{I}_1
  \oplus (n_2 + n_3) \cdot \mathcal{I}_2$, and consequently
  $d(\mathcal{M}) = \max \{n_1, n_2\} + n_3$.  Since $\dim_{\Q_p}
  \mathcal{M} = \dim_{\Z_p} M = n_1 + (p-1)n_2 + p n_3$, the
  assumption $1 + d(\mathcal{M}) \geq \dim_{\Q_p} \mathcal{M}$ leads
  to the inequality \eqref{equ:module}.
\end{proof}

For later use we record in Table~\ref{table} the numerical
consequences of this lemma.

\begin{table}[H]
  \begin{center}
    \begin{tabular}{|c|c!{\vrule width 1pt}c|c|c!{\vrule width
          1pt}c|}\hline
      range of $p$ & conditions & $n_1$ & $n_2$ & $n_3$ & case label \\ \hline
      $p=2$ & $n_1 > n_2$ & $\geq 2$ & $0$ & $0$ & (T~2.1) \\
      & & $\geq 2$ & $1$ & $0$ & (T~2.2) \\
      & & $\geq 1$ & $0$ & $1$ & (T~2.3) \\
      & $n_1 \leq n_2$ & $0$ & $\geq 2$ & $0$ & (T~2.4) \\
      & & $1$ & $\geq 1$ & $0$ & (T~2.5) \\
      & & $0$ & $\geq 0$ & $1$ & (T~2.6) \\ \hline
      $p=3$ & $n_1 > n_2$ & $\geq 2$ & $0$ & $0$ & (T~3.1) \\
      & $n_1 \leq n_2$ & $0$ & $1$ & $0$ & (T~3.2) \\ \hline
      $p \geq 5$ & -- & $\geq 2$ & $0$ & $0$ & -- \\ \hline
    \end{tabular}
  \end{center}

  \smallskip

  \caption{Values of $n_1,n_2,n_3 \in \N_0$ satisfying
    the inequality~\eqref{equ:module} and  the condition $n_1 + (p-1)n_2 + p
    n_3 = \dim_{\Q_p} \mathcal{M} \geq 2$.}
  \label{table}
\end{table}


\section{The classification -- proof of
  Theorem~\ref{thm:main_theorem}}

In this section we prove Theorem~\ref{thm:main_theorem}.  First we
present an argument applying to the `generic' case $p \geq 5$.  The
exceptional primes $2$ and $3$ can be treated in a similar way, but
lead to extra complications.  We leave it as an easy exercise to show
that the groups listed in Theorem~\ref{thm:main_theorem} have indeed
constant generating number on open subgroups, apart from a brief
comment on the case (3) in the appropriate subsection below. Clearly,
the assertions of Theorem~\ref{thm:main_theorem} for virtually
pro-cyclic pro-$p$ groups are correct, hence in the proofs below we
only consider groups of dimension at least~$2$.

\subsection{} First suppose that $p \geq 5$. Let $G$ be a finitely
generated pro-$p$ group with constant generating number $d = d(G) \geq
2$ on open subgroups.  Then $G$ has finite rank and thus admits an
open normal subgroup $H$ which is saturable.  We need to show that $G$
is isomorphic to one of the groups listed in (1) and (2) of
Theorem~\ref{thm:main_theorem}, which are saturable.

We argue by induction on $\lvert G : H \rvert$.  If $G = H$, the
claim follows from Corollary~\ref{cor:sat_groups}.  Hence suppose
that $\lvert G : H \rvert \geq p$.  Running along a subnormal
series from $H$ to $G$, we may assume by induction that $\lvert G
: H \vert = p$.  This means that $G = \langle z \rangle H$ with
$G/H = \langle \overline{z} \rangle \cong C_p$.  Moreover, by
Corollary~\ref{cor:sat_groups}, there are essentially two
possibilities for the group $H$.

\smallskip

\noindent \textit{Case 1:} $H$ is abelian.  We regard $H$ as a $\Z_p
\langle \overline{z} \rangle$-module and put $\mathcal{M} := \Q_p
\otimes H$.  It is easily seen that $G$ admits an open subgroup $K$,
containing $z$, which requires at most $1 + d(\mathcal{M})$
generators.  This implies $1 + d(\mathcal{M}) \geq d(K) = d =
\dim_{\Q_p} \mathcal{M}$.  Thus Lemma~\ref{lem:module} and
Table~\ref{table}, for $p=5$, show that $z$ acts trivially on $H$.
Hence $G$ is abelian.  Clearly, this implies that $G \cong \Z_p^d$.

\smallskip

\noindent \textit{Case 2:} $H$ is not abelian.  Then
Corollary~\ref{cor:sat_groups} shows that $H$ is of the form $\langle
y \rangle \ltimes A $, where $\langle y \rangle \cong \Z_p$, $A \cong
\Z_p^{d-1}$ is abelian, normal in $H$ and conjugation by $y$ on $A$
corresponds to scalar multiplication by $1+p^s$ for some $s \geq 1$.
Note that $A$ is characteristic in $H$, hence normal in $G$.  (Indeed,
$A$ is the isolator of the commutator subgroup $[H,H]$ in $H$.) We
consider two separate cases, $z^p \in A$ and $z^p \not \in A$.

\noindent \textit{Case 2.1:} $z^p \in A$.  We contend that the group
$\langle z \rangle A$ has constant generating number $d-1$ on open
subgroups.  Indeed, given an open subgroup $B$ of $\langle z \rangle
A$, we choose a normal open subgroup $N$ of $G$ and $r \in \N$ such
that $d(B) = d(BN/N)$ and $y^{p^r} \in N$.  Then $K := \langle B \cup
\{ y^{p^r} \} \rangle$ is an open subgroup of $G$ with $d(K) \leq d(B)
+ 1$.  But $KN/N = BN/N$ shows that $d(K) \geq d(B)$ and, moreover,
that any minimal generating set of $B$ extends to a minimal generating
set of $K$.  Since $B$ is properly contained in $K$, we deduce that
$d(K) = d(B) + 1$.  From this we get $d(B) = d(K) - 1 = d-1$, as
claimed.

By Case $1$, we conclude that $\langle z \rangle A \cong
\Z_p^{d-1}$ is abelian.  Since $y$ acts as multiplication by $1+p^s$
on $A$ it must also act in the same way on $\langle z \rangle A$, and
hence $G \cong H$ is of the required form.

\noindent \textit{Case 2.2:} $z^p \not \in A$.  Without loss of
generality we may assume that $z^p = y^{p^k} a$ with $k \in \N_0$ and
$a \in A$.  First consider the case $k \geq 1$.  Since $A$ is
characteristic in $H$, $z$ acts on $H/A \cong \Z_p$.  As $p>2$, this
action of $z$ is trivial, i.e.\ $[y,z] \in A$.  Put $z_1 := z^{-1}
y^{p^{k-1}}$.  Then $z_1^p = (z^{-1} y^{p^{k-1}})^p \equiv z^{-p}
y^{p^k} \equiv 1$ modulo $A$, i.e.\ $z_1^p \in A$.  Replacing $z$ by
$z_1$ we return to Case 2.1.

Finally suppose that $k = 0$.  Replacing $y$ by $ya$, we may assume
that $z^p = y$.  Since $A$ is characteristic in $H$, we conclude that
$G = \langle z \rangle \ltimes A$.  Regarding $A$ as a $\Z_p \langle z
\rangle$-module, we put $\mathcal{M} := \Q_p \otimes A$.  It is easily
seen that $G$ admits an open subgroup $K$, containing $z$, which
requires at most $1 + d(\mathcal{M})$ generators.  Since $G$ has
constant generating number $d$ on open subgroups, it follows that
$d(\mathcal{M}) = d-1 = \dim_{\Q_p} \mathcal{M}$.  This implies that
$z$ must operate as multiplication by a scalar $\lambda \in \Z_p^*$ on
$A$.  Since $p>2$, we can replace $z$ by a suitable power of itself
which acts as multiplication by $1+p^s$ for some $s \in \N$.

\subsection{} Next we deal with the `exceptional' case $p=3$.  A short
argument shows that the pro-$3$ group $G_0 = \langle w \rangle \ltimes
B$ of maximal class, defined in (3) of Theorem~\ref{thm:main_theorem},
has constant generating number $2$ on open subgroups.  Indeed, every
open subgroup of $G_0$ requires at least $2$ generators, because $G_0$
is not virtually pro-cyclic.  Moreover, putting $\pi := \omega-1$, the
filtration $\pi^k B$, $k \in \N_0$, of the subgroup $B$ is tight: the
index $\lvert \pi^k B : \pi^{k+1} B \rvert$ between any two
consecutive terms is equal to $3$.  Hence the open subgroups of $G_0$
not contained in $B$ are of the form $\langle w_1 \rangle \ltimes
\pi^k B$ for some $w_1 \in wB$ of order $3$ and $k \in \N_0$. These
groups are all isomorphic to $G_0$ and require no more than $2$
generators.  Open subgroups contained in $B \cong \Z_3^2$ obviously
require no more than $2$ generators.

Now let $G$ be a finitely generated pro-$3$ group with constant
generating number $d = d(G) \geq 2$ on open subgroups.  Similarly as
in the generic case (i.e.\ $p \geq 5$), Corollary~\ref{cor:sat_groups}
provides the starting point for an induction argument.  We consider an
open normal subgroup $H$ of index $3$ in $G$, and we write $G =
\langle z \rangle H$.  Inductively, we assume that $H$ is isomorphic
to one of the groups listed in (1), (2) and (3) of
Theorem~\ref{thm:main_theorem}, and we need to prove that the same
holds true for $G$.  As before, there are several cases to consider.

\smallskip

\noindent \textit{Case 1:} $H$ is abelian.  Similarly as in the
generic case, we regard $H$ as a $\Z_3 \langle \overline{z}
\rangle$-module and put $\mathcal{M} := \Q_3 \otimes H$.  Since $G$
admits an open subgroup $K$, containing $z$, which requires at most $1
+ d(\mathcal{M})$ generators, we have $1 + d(\mathcal{M}) \geq d(K) =
d = \dim_{\Q_3} \mathcal{M}$.  Thus Lemma~\ref{lem:module} and
Table~\ref{table} for $p=3$ show that there are two possibilities.  If
$z$ acts trivially on $H$, corresponding to case (T~3.1) in
Table~\ref{table}, then $G \cong \Z_3^d$ is abelian as in the generic
case.  Now suppose that $z$ does not act trivially on $H$.  Then $H$
is indecomposable as a $\Z_3 \langle \overline{z} \rangle$-module and
of $\Z_3$-dimension $2$, corresponding to case (T~3.2) in
Table~\ref{table}.  Since $H$ contains no non-trivial elements which
are fixed under conjugation by $z$, we must have $z^3 =
1$. Consequently, $G$ is isomorphic to the pro-$3$ group of maximal
class described in (3) of Theorem~\ref{thm:main_theorem}.

\smallskip

\noindent \textit{Case 2:} $H$ is not abelian.  First we consider the
`new' case that $H = \langle w \rangle \ltimes B$ is a pro-$3$ group
of maximal class, as described in (3) of
Theorem~\ref{thm:main_theorem}.  Being the unique maximal torsion-free
subgroup of $H$, the subgroup $B$ is characteristic in $H$ and hence
normal in $G$.  Because $H/B \cong C_3$ does not admit automorphisms
of order $3$, the element $z$ acts trivially on $H/B$ and thus
conjugation by $z$ preserves the $\Z_3[\omega]$-structure of $B$.  But
$B$ is a free $\Z_3[\omega]$-module and its automorphisms are given by
multiplication by units of the ring $\Z_3[\omega]$.  The torsion
subgroup of the unit group $\Z_3[\omega]^*$ is $\langle \omega \rangle
\cong C_3$.  Replacing $z$ by $zw$ or $zw^2$, if necessary, we may
assume that $z$ acts trivially on~$B$.  But then $\langle z \rangle B
= \Cen_G(B)$ is an abelian normal subgroup of index $3$ in $G$ and,
returning to Case 1, we find out that $G$ is, in fact, isomorphic to
$H$ and covered by (3) of Theorem~\ref{thm:main_theorem}.

It remains to consider, similarly as in the generic case (i.e.\ $p
\geq 5$), the situation for $H$ of the form $\langle y \rangle \ltimes
A$, where $\langle y \rangle \cong \Z_3$, $A \cong \Z_3^{d-1}$ is
abelian and conjugation by $y$ on $A$ corresponds to scalar
multiplication by $1+3^s$ for some $s \geq 1$.  As before, $A$ is
characteristic in $H$, hence normal in $G$, and we consider two cases.

\noindent \textit{Case 2.1:} $z^3 \in A$.  As explained in the generic
case, the group $\langle z \rangle A $ has constant generating number
$d-1$ on open subgroups.  By Case~$1$, one of the following holds:
$\langle z \rangle A \cong \Z_3^{d-1}$ is abelian, which can be dealt
with as before, or $\langle z \rangle A$ is a pro-$3$ group of maximal
class.  We now show that the latter option leads to a contradiction
and hence does not occur.  Without loss of generality, we may assume
that $z$ acts on $A = \Z_3 + \Z_3 \omega$ as multiplication by a third
root of unity $\omega$.  Then $yz$ operates on $A$ as multiplication
by $\omega (1+3^s)$.  Thus $\langle yz, a \rangle$, for any
non-trivial element $a \in A$ is an open subgroup of $G$ which
requires less than $3$ generators, a contradiction.

\noindent \textit{Case 2.2:} $z^3 \not \in A$.  The argument goes
through as in the generic case (i.e.\ $p \geq 5$).

\subsection{} Finally, we consider the remaining `exceptional' case
$p=2$.  Let $G$ be a finitely generated pro-$2$ group with constant
generating number $d = d(G) \geq 2$ on open subgroups.  Similarly as
before, suppose that $H$ is an open normal subgroup of index $2$ in $G
= \langle z \rangle H$ and that we inductively understand $H$.  There
are, once more, several cases to consider.  In several places we will
use the basic equality $\log_2 \lvert G : G^2 \rvert = d(G) = d$,
which relies on the fact that $G^2 = \langle g^2 \mid g \in G \rangle$
is equal to the Frattini subgroup $\Phi(G)$ of $G$.

\smallskip

\noindent \textit{Case 1:} $H \cong \Z_2^d$ is abelian.  Similarly as
in the situations considered previously (i.e.\ for $p \geq 5$ and
$p=3$), we regard $H$ as a $\Z_2 \langle \overline{z} \rangle$-module
and put $\mathcal{M} := \Q_2 \otimes H$.  As before, one shows that $1
+ d(\mathcal{M}) \geq \dim_{\Q_2} \mathcal{M}$.  Thus
Lemma~\ref{lem:module} and Table~\ref{table} for $p=2$ imply that
there are six possibilities, labelled (T~2.1) up to (T~2.6).

Clearly, the case (T~2.1) can be dealt with as before: $G$ is seen to
be abelian.  Next we show that the case (T~2.2) leads to a contradiction.
Indeed, if $z$ acts on $H = \langle x_1, \ldots, x_d \rangle$
according to $x_i^z = x_i$ for $1 \leq i \leq d-1$ and $x_d^z =
x_d^{-1}$, where $d \geq 3$, then $\log_2{\lvert G : G^2 \rvert} = d(G) = d$
implies that $z^2 \not \in H^2$. Thus we may assume without loss of
generality that $z^2 = x_1$. But then $K := \langle z, x_2, \ldots,
x_{d-2}, x_{d-1} x_d \rangle$ is an open subgroup of $G$ with
$d(K) = d-1$, a contradiction.

Now we show that the case (T~2.3) also leads to a contradiction.
Indeed, suppose that $z$ acts on $H = \langle x_1, \ldots, x_d
\rangle$ according to $x_1^z = x_2$, $x_2^z = x_1$ and $x_i^z = x_i$
for $3 \leq i \leq d$, where $d \geq 3$.  Since $z^2$ commutes with
$z$, we have $z^2 \in \langle x_1 x_2, x_3, \ldots, x_d \rangle$.  If
$z^2 \not \in \langle x_1 x_2 \rangle$, then $G$ admits an open
subgroup $K$ generated by $z$, $x_1$ and $d-3$ of the elements $x_3,
\ldots, x_d$, in contradiction to the requirement $d(K)=d$.  Now
suppose that $z^2 \in \langle x_1 x_2 \rangle$.  Then $z^2 = (x_1
x_2)^\lambda$ for $\lambda \in \Z_2$, and $z_1 := z x_1^{-\lambda}$
satisfies $z_1^2 = z^2 x_2^{-\lambda} x_1^{-\lambda} = 1$.  Replacing
$z$ by $z_1$, if necessary, we may assume that $z^2 = 1$.  But then
$G$ admits the open subgroup $K = \langle z, x_1 x_2, x_1^2, x_3,
\ldots, x_d \rangle$ with $d(K) = d+1$, a contradiction.


Next we show that the case (T~2.4) again leads to a contradiction.
Indeed, if $z$ acts on $H = \langle x_1, \ldots, x_d \rangle$
according to $x_i^z = x_i^{-1}$ for $1 \leq i \leq d$, where $d \geq
2$, then $\log_2{\lvert G : G^2 \rvert} = d(G) = d$ implies that $z^2 \not \in
H^2$.  But $z$ commutes with $z^2 \in H$ so that $z^2 =1$, a
contradiction.

Now we consider the case (T~2.5): $z$ acts on $H = \langle x_1,
\ldots, x_d \rangle$ according to $x_1^z = x_1$ and $x_i^z = x_i^{-1}$
for $2 \leq i \leq d$, where $d \geq 2$.  Since $\log_2{\lvert G : G^2 \rvert}
= d(G) = d$, we have $z^2 \not \in H^2$ and we may assume that $z^2 =
x_1$.  This shows that $G$ is isomorphic to the group in (4)
of Theorem~\ref{thm:main_theorem}.

Finally, we treat the case (T~2.6): $z$ acts on $H = \langle x_1,
\ldots, x_d \rangle$ according to $x_1^z = x_2$, $x_2^z = x_1$ and
$x_i^z = x_i^{-1}$ for $3 \leq i \leq d$, where $d \geq 2$.  Since
$z^2$ commutes with $z$, we have $z^2 \in \langle x_1 x_2 \rangle$.
We can now argue similarly as in the case (T~2.3): $z^2 = (x_1
x_2)^\lambda$ for $\lambda \in \Z_2$, and $z_1 := z x_1^{-\lambda}$
satisfies $z_1^2 = z^2 x_2^{-\lambda} x_1^{-\lambda} = 1$.  Replacing
$z$ by $z_1$, if necessary, we may assume that $z^2 = 1$.  But then
$G$ admits the open subgroup $K = \langle z, x_1 x_2, x_1^2, x_3,
\ldots, x_d \rangle$ with $d(K) = d+1$, a contradiction.


\smallskip

\noindent \textit{Case 2:} $H$ is not abelian. First we deal with the
case that $H$ (and hence $G$) is virtually abelian; this means that
$H$ is of the form described in (4) of Theorem~\ref{thm:main_theorem}.
In this case we show that $G$ has an abelian normal subgroup of index
$2$ so that we can return to Case~1.  Indeed, let $A \cong \Z_2^d$ be
a maximal open abelian normal subgroup of $G$.  For a contradiction,
we may assume that $\lvert G:A \rvert = 4$.  There are two
possibilities: $G/A$ is either cyclic or a Klein $4$-group.

If $G/A = \langle \overline{z} \rangle$ is cyclic of order $4$, then
$\overline{z}$ acts on $A$ as an element of order $4$.  Consequently,
the $\Q_2 \langle \overline{z} \rangle$-module $\mathcal{M} := \Q_2
\otimes A$ admits at least one irreducible submodule $\mathcal{J}$ of
dimension $2$ featuring the eigenvalues $\pm \sqrt{-1}$ for
$\overline{z}$.  This means that $\mathcal{M}$ requires at most $d-1$
generators.  Moreover, by Case~1, the open subgroup $\langle z^2
\rangle A$ must be isomorphic to the group described in (4) of
Theorem~\ref{thm:main_theorem}.  The structure of this group shows
that $z^4 \in \mathcal{M} \setminus \mathcal{J}$, thus providing one
extra module generator.  Consequently, $G$ admits an open subgroup
$K$, containing $z$, which requires at most $1 + (d-1) - 1 = d-1$
generators, a contradiction.

Now suppose that $G/A = \langle \overline{w}, \overline{z} \rangle
\cong C_2 \times C_2$.  Since $A$ is a maximal open abelian normal
subgroup of $G$, the isomorphism types of the three larger subgroups
$\langle w \rangle A$, $\langle z \rangle A$ and $\langle wz \rangle
A$ of $G$ are limited by Case~1: each of these three groups is
isomorphic to the group described in (4) of
Theorem~\ref{thm:main_theorem}.  This shows that the eigenvalues
associated to the action of each of the three elements $\overline{w}$,
$\overline{z}$, $\overline{w} \overline{z}$ on the $\Q_2$-vector space
$\Q_2 \otimes A$ are: $1$, with multiplicity $1$, and $-1$, with
multiplicity $d-1$. Moreover, $w^2$, $z^2$, $(wz)^2$ each give a
generator of $A$, i.e.\ an element of $A \setminus A^2$.  The action
of the abelian group $\langle \overline{w}, \overline{z} \rangle$ on
$\Q_2 \otimes A$ can, of course, be represented by diagonal matrices.
The described eigenvalue spectra of $\overline{w}$, $\overline{z}$,
$\overline{w} \overline{z}$ thus show that $\dim_{\Q_2} \Q_2 \otimes A
= 3$ and that the action of $\overline{w}$, $\overline{z}$ and
$\overline{w} \overline{z}$ on $\Q_2 \otimes A$ with respect to a
suitable $\Q_2$-basis is given by the matrices
$$
\begin{pmatrix}
  1 & 0  & 0 \\
  0 & -1 & 0 \\
  0 & 0  & -1
\end{pmatrix}, \quad
\begin{pmatrix}
 -1 & 0  & 0 \\
  0 & 1  & 0 \\
  0 & 0  & -1
\end{pmatrix}, \quad
\begin{pmatrix}
 -1 & 0  & 0 \\
  0 & -1 & 0 \\
  0 & 0  & 1
\end{pmatrix}.
$$
But then $w^2$, $z^2$ and $(wz)^2$, which as explained each give rise to a
generator of $A$, actually form a basis for $\Q_2 \otimes A$.  This
shows that $G = \langle w,z \rangle$ requires only $2$ generators, a
contradiction.

It remains to consider the situation where $H$ is not virtually
abelian, and thus isomorphic to one of the groups in (2) of
Theorem~\ref{thm:main_theorem}: $H$ is of the form $\langle y \rangle
\ltimes A $, where $\langle y \rangle \cong \Z_2$, $A \cong
\Z_2^{d-1}$ is abelian, normal in $H$ and conjugation by $y$ on $A$
corresponds to scalar multiplication by $\pm (1+2^s)$ for some $s \geq
2$.  As in the previous cases (i.e.\ $p \geq 5$ and $p=3$), $A$ is
characteristic in $H$, hence normal in $G$.  We consider two separate
cases.

\noindent \textit{Case 2.1:} $z^2 \in A$.  As before, the group
$\langle z \rangle A$ has constant generating number $d-1$ on its open
subgroups.
If the group $\langle z \rangle A$ is not abelian, then it is of the
form described in (4) of Theorem~\ref{thm:main_theorem}, with $z$
acting as $\overline{z}$ of order $2$ on $A$.  The $\Q_2 \langle
\overline{z} \rangle $-module $\mathcal{M} := \Q_2 \otimes A$ is isomorphic to
$\mathcal{I}_1 \oplus (d-2) \mathcal{I}_2$, where we are using the
notation introduced in the proof of Lemma~\ref{lem:module}.  Thus
$d(\mathcal{M}) = d-2$.  We observe that this number does not change,
if we regard $\mathcal{M}$ as a $\Q_2 \langle yz \rangle$-module
rather than a $\Q_2 \langle z \rangle$- or $\Q_2 \langle \overline{z}
\rangle$-module.  This shows that $G$ contains an open subgroup $K$,
containing $yz$, such that $d(K) = d-1$, a contradiction.  Thus
$\langle z \rangle A \cong \Z_2^{d-1}$ is abelian.  Since $y$ acts as
multiplication by $\pm (1+2^s)$ on $A$ it must also act in the same
way on $\langle z \rangle A$, and hence $G \cong H$ is of the required
form.

\noindent \textit{Case 2.2:} $z^2 \not \in A$.  Without loss of
generality we may assume that $z^2 = y^{2^k} a$ with $k \in \N_0$ and
$a \in A$.  First consider the case $k \geq 1$.  Since $A$ is
characteristic in $H$, $z$ acts on $H/A \cong \Z_2$ and there are two
possibilities: either $y^z \equiv y$ or $y^z \equiv y^{-1}$ modulo
$A$. If $z$ acts trivially on $H/A$, we return to Case 2.1, just as in
the generic case (i.e.\ $p=5$).  We claim that the second possibility
does not occur.  Indeed, suppose that $y^z \equiv y^{-1}$ modulo $A$.
Then, for every $a \in A$, we have
$$
\pm (1+2^s) a^z = (a^y)^z = (a^z)^{y^z} = (a^z)^{y^{-1}} = \pm
(1+2^s)^{-1} a^z.
$$
This implies that $(1+2^s)^2 = 1$, in contradiction to $s \geq 2$.

Finally suppose that $k = 0$.  Similarly as in the generic case
(i.e.\ $p=5$), we may assume without loss of generality that $z^2 = y$
and we subsequently reduce to $G = \langle z \rangle \ltimes A$, with
$z$ operating as multiplication by a scalar $\lambda \in \Z_2^*$ on $A
\cong \Z_2^{d-1}$.  This implies that $z^2 = y$ acts on $A$ as
multiplication by $\lambda^2 = 1+2^s$ where $s \geq 3$.  Hence, in a
final step we can replace $z$ by a suitable power of itself which acts
as multiplication by $\pm (1+2^{s-1})$.

\bigskip

\ackn The second author started his investigation of Iwasawa's
question as a PhD student at Binghamton University.  He thanks his
adviser Prof.\ Marcin Mazur for continuous support and helpful
discussions.

\bibliographystyle{plain}

\end{document}